\documentclass[letterpaper, 10 pt, conference]{ieeeconf}  % Comment this line out if you need a4paper
\IEEEoverridecommandlockouts                              % This command is only needed if 
\usepackage{amsmath}
\usepackage{amssymb}
\usepackage{bm}
\usepackage{comment}
\usepackage{cases}
\usepackage[pdftex]{graphicx}
\usepackage{cite}
\usepackage{algorithm}
\usepackage{algpseudocode}
\usepackage{subfigmat}
\usepackage{here}

\newtheorem{thm}{Theorem}
\newtheorem{pbm}{Problem}
\newtheorem{col}{Corollary}

\newcommand{\tr}{\mathsf{T}}

\AtBeginDocument{
	\abovedisplayskip     =0.5\abovedisplayskip
	\abovedisplayshortskip=0.5\abovedisplayshortskip
	\belowdisplayskip     =0.5\belowdisplayskip
	\belowdisplayshortskip=0.5\belowdisplayshortskip}

\title{\LARGE \bf
Data-Driven Koopman Controller Synthesis\\ Based on the Extended 
$\mathcal{H}_2$ Norm Characterization
}

\author{Daisuke Uchida$^{1}$, Atsushi Yamashita$^{1}$,
	and Hajime Asama$^{1}$% <-this % stops a space
\thanks{$^{1}$Department of Precision Engineering, 
		School of Engineering,
        The University of Tokyo, 
        Hongo 7-3-1, Bunkyo, Tokyo, Japan
        {\tt\small \{uchida, yamashita, asama\}@robot.t.u-tokyo.ac.jp}}%
}

\begin{document}

\maketitle
\thispagestyle{empty}
\pagestyle{empty}

%%%%%%%%%%%%%%%%%%%%%%%%%%%%%%%%%%%%%%%%%%%%%%%%%%%%%%%%%%%%%%%%%%%%%%%%%%%%%%%%
\begin{abstract}
This paper presents a new data-driven controller synthesis based on the Koopman operator and the extended $\mathcal{H}_2$ norm characterization of discrete-time linear systems.
We model dynamical systems as polytope sets
which are
derived from multiple data-driven linear models obtained
by the finite approximation of the Koopman operator 
and then used to design robust feedback
controllers combined with the $\mathcal{H}_2$ norm characterization. 
The use of the $\mathcal{H}_2$ norm characterization is aimed to deal with the model uncertainty that arises due to the nature of the data-driven setting of the problem.
The effectiveness of the proposed controller synthesis is 
investigated through numerical simulations.
\end{abstract}

%%%%%%%%%%%%%%%%%%%%%%%%%%%%%%%%%%%%%%%%%%%%%%%%%%%%%%%%%%%%%%%%%%%%%%%%%%%%%%%%
\section{INTRODUCTION}
The Koopman operator is an infinite-dimensional linear operator describing the evolution of so-called observable functions of underlying
dynamical systems.
Recently several data-driven techniques 
\cite{Applied_Koopmanism,Spectral_analysis_of_nonlinear_flows,schmid_2010,Tu_2014,Williams2015}
which estimate the spectral properties of the Koopman operator
have gained popularity in various fields.
The data-driven Koopman operator framework can lift a nonlinear system to a linear setting
in a data-driven manner. 
On the other hand the Koopman operator is infinite dimensional, and it is often necessary for engineering and scientific applications to approximate it
by a finite dimensional one. 
There has been considerable effort
to design appropriate observable functions that enable to numerically approximate the Koopman operator
\cite{Kernel-DMD,Learning_DNN_ACC2019,Dictionary_learning,Learning_Koopman_Invariant_Subspaces}.
%This can be achieved by the construction of an invariant subspace of observable functions under the action of the Koopman operator, and there have been several research developing techniques to design appropriate observable functions that enable to numerically approximate the Koopman operator
%\cite{Kernel-DMD,Learning_DNN_ACC2019,Dictionary_learning,Learning_Koopman_Invariant_Subspaces}. 
These studies offered promising techniques to design observable functions using several useful concepts in other fields such as machine learning. Although they advanced the progress of the data-driven Koopman operator theory, it is still a grand challenge to design observable functions which have the high accuracy of prediction for a long period of time with moderate computational costs and complexity.

Design of observable functions is an important issue also for the Koopman operator-based controller design. It is necessary 
for better data-driven controllers 
to specify observable functions so that data-driven control models can reconstruct the behavior of underlying dynamical systems with high accuracy while remaining the moderate degrees and complexity.
While the Koopman operator theory was originally developed in the context of autonomous systems, it can be also
extended to non autonomous settings, where systems have inputs\cite{DMDc,EDMDc,EDMD_actuated_system}.
On the basis of these formulations,
problems on the data-driven controller design have been investigated by several studies\cite{Koopman_control_PLOS_ONE,KORDA_Koopman_MPC,Arbabi_Koopman_MPC_nonlinear_flows,PEITZ_switched_control,Peitz_interpolated_generator_MPC,Kaiser_eigenfunction_control}.
Promising mehods have been proposed to construct data-driven linear models
via the Koopman operator, which were followed by the model predictive control (MPC) design\cite{KORDA_Koopman_MPC,Arbabi_Koopman_MPC_nonlinear_flows}.
In \cite{PEITZ_switched_control} an alternative approach with low dimensional switched-systems was developed 
in order to reduce the data requirement, which was further extended to include continuous inputs by interpolation\cite{Peitz_interpolated_generator_MPC}.
Many of these studies are based on Extended Dynamic Mode Decomposition (EDMD)\cite{Williams2015} to obtain data-driven linear models. 
In addition to these efforts there is also another research direction, eigenfunction control\cite{Kaiser_eigenfunction_control},
where the Koopman eigenfunctions estimated from data are directly used to design controllers in a linear fashion.

In the data-driven Koopman controller synthesis, 
linear control models completely reconstruct the behavior of the underlying nonlinear
systems if we can design appropriate observable functions and collect a
sufficient amount of data required for approximating the Koopman operator.
However it is often the case in real situations that we
have difficulty finding such ideal observable functions
and face some limitation on collecting data, which leads to
the model uncertainty in our data-driven control models.

In this paper, we explicitly deal with this model uncertainty utilizing the extended $\mathcal{H}_2$ norm characterization of discrete-time linear time-invariant systems\cite{Extended_H2_H_inf}. It is shown that the linear robust control theory can be easily incorporated with the data-driven nonlinear systems control since the data-driven Koopman operator framework reduces underlying nonlinear systems to linear ones
in a data-driven manner. 
Specifically we form a polytope set as a control model
by utilizing multiple data-driven models obtained via the EDMD algorithm, and design a robust feedback controller based on the linear matrix inequality (LMI) conditions of the extended $\mathcal{H}_2$ norm.

This paper is organized as follows. In Section \ref{sec2} we outline the
data-driven Koopman operator theory in the context of controller synthesis. Section \ref{sec3} proposes a new controller synthesis based on the extended $\mathcal{H}_2$ norm condition of the data-driven model constructed in the previous section. Numerical examples are shown in Section \ref{sec4}.

\section{KOOPMAN OPERATOR FRAMEWORK FOR THE CONTROLLER SYNTHESIS}
\label{sec2}
\subsection{Koopman Operator for Systems with Inputs}
First we outline the data-driven Koopman operator theory for systems with input signals.
Consider a nonlinear dynamical system described as follows:
%\vspace{-4mm}
\begin{align}
	\dot{\bm{x}} = \bm{f}(\bm{x}, \bm{u}),\ \ \ 
	\bm{x}\in \mathcal{X}\subset \mathbb{R}^n,\ \ 
	\bm{u}\in \mathcal{U}\subset \mathbb{R}^p,
	\label{eq. continuous underlting system}
	\vspace*{-2mm}
\end{align}
where $\bm{x}$ and $\bm{u}$ denote the state and the input, respectively.
Discretizing the system with a fixed time interval $\Delta t$ leads to a discrete-time system:
%\vspace{-4mm}
\begin{align}
\label{eq. dicritized underlying system}
&\bm{x}_{k+1} = \bm{F}(\bm{x}_k,\bm{u}_k),\
\left(
\bm{x}_{k} \hspace{-1mm}=\hspace{-1mm} 
\bm{x}(k\Delta t),
\bm{u}_k \hspace{-1mm}=\hspace{-1mm} 
\bm{u}(k\Delta t)
\right).&
\end{align}

In this paper it is assumed that we have no explicit knowledge about
the underlying equations of the system (\eqref{eq. continuous underlting system} and \eqref{eq. dicritized underlying system}).
Instead a data-driven 
controller synthesis, which constructs a linear model from data, is developed by utilizing the Koopman operator.
Let $g$ denote an observable function in some function space $\mathcal{G}$ s.t.
%\vspace{-2mm}
\begin{align}
	g:\mathcal{X}\times 
	\mathcal{U}
	\rightarrow \mathbb{R},\ \ \ 
	g\in \mathcal{G}.
	%:
	%(\bm{x}_{k}, (\bm{u}_k)_{k=0}^{\infty})
	%\mapsto 
	%g(\bm{x}_k, (\bm{u}_k)_{k=0}^{\infty}),
\end{align}

The Koopman operator $\mathcal{K}:\mathcal{G}\rightarrow \mathcal{G}$ corresponding to $\bm{F}$ is defined 
as follows:
%\vspace{-2mm}
\begin{align}
	&(\mathcal{K}g)(\bm{x}_k, \bm{u}_k) = 
	g(\bm{F}(\bm{x}_{k}),\bm{u}_{k+1})
	=
	g(\bm{x}_{k+1},\bm{u}_{k+1}).&
\end{align}

The physical and intuitive meanings of the Koopman operator may slightly change according to 
how the evolution of the input $\bm{u}_k$
is governed, e.g.
exogenous random variables, closed loop
signals, and so on\cite{EDMDc}.

Let $g_i\in \mathcal{G}$ for $i=1,\cdots N$.
We consider to approximate the infinite dimensional Koopman operator
$\mathcal{K}$ by the finite dimensional one
$\bm{\mathcal{K}}\in \mathbb{R}^{N\times N}$
that is derived from $g_i$:
%\vspace{-2mm}
%We introduce the finite dimensional approximation 
%$\bm{\mathcal{K}}\in \mathbb{R}^{N\times N}$ 
%of the Koopman operator $\mathcal{K}$:
\begin{align}
	&
	[
		\left( \mathcal{K}g_1 \right)(\bm{x}_k,\bm{u}_k)
		\cdots 
		\left( \mathcal{K}g_{N} \right)(\bm{x}_k,\bm{u}_k)
	]^\tr
	&\nonumber
\\ 
	&
	\hspace{33mm}
	=
	\bm{\mathcal{K}}\bm{g}
	(\bm{x}_k,\bm{u}_k)
	+
	\bm{r}(\bm{x}_k,\bm{u}_k),&
	\label{eq. residual 1}
\end{align}
where $\bm{r}(\bm{x}_k,\bm{u}_k)\in \mathbb{R}^N$ represents the residual and $\bm{g}$ denotes a vector-valued observable function
defined as:
%\vspace{-2mm}
\begin{align}
	\bm{g}(\bm{x}_k,\bm{u}_k):=&
	[ g_1(\bm{x}_k,\bm{u}_k) \cdots 
	g_{N}(\bm{x}_k,\bm{u}_k) ]^\mathsf{T}.
	&
\end{align}

We deal with a problem to control nonlinear systems via the observable functions $g_i$ and the finite dimensional Koopman operator $\bm{\mathcal{K}}$. A data-driven procedure is adopted to obtain linear models for controller synthesis. Observable functions 
may be either physical measurements (outputs), the state itself, or 
functions which were designed by users.

\subsection{Construction of Linear Models via Extended Dynamic Mode Decomposition}
In the context of the data-driven Koopman operator framework, our proposed controller synthesis 
constructs linear control models purely from data
in the same way developed in \cite{KORDA_Koopman_MPC}.
In order to obtain finite-dimensional linear systems
represented by matrices
we impose the linearity on inputs, i.e., observable functions are of the form:
%\vspace{-2mm}
\begin{align}
\bm{g}(\bm{x}_k, \bm{u}_k)=
[ g_1(\bm{x}_k)\cdots g_N(\bm{x}_k)\ \bm{u}_k^\tr  ]^\mathsf{T},
\end{align}
%\vspace{-2mm}
where $g_i:\mathcal{X}\rightarrow \mathbb{R}$ (for $i=1,\cdots,N$) denote observable functions defined on the state space $\mathcal{X}$. 
On the basis of this formulation a discrete-time linear model in the form:
%\vspace{-2mm}
\begin{align}
	\bm{g}(\bm{x}_{k+1}) =& \bm{A}\bm{g}(\bm{x}_k) + \bm{B}\bm{u}_k,&
	\label{eq. linear predictor}
\end{align}
%\vspace{-2mm}
where
$\bm{g}(\bm{x}_k):=[ g_1(\bm{x}_k) \cdots g_N(\bm{x}_k) ]^\mathsf{T},$
is sought as follows. 
First consider a data set:
%\vspace{-2mm}
\begin{align}
&
D:=\{ \bm{x}_k,\bm{u}_k \}_{k=1}^{M+1},
&
\\
&
\bm{x}_{k+1} = \bm{F}(\bm{x}_k, \bm{u}_k)
\ 
\text{for }k=1,\cdots,M.
&
\end{align}

Next define the matrices:
%\vspace{-2mm}
\begin{align}
	\bm{G}:=&[\bm{g}(\bm{x}_1)\cdots \bm{g}(\bm{x}_M)]
	\in \mathbb{R}^{N\times M},&
\end{align}
\begin{align}
	\bm{\hat{G}}:=&
	[\bm{g}(\bm{x}_2)\cdots \bm{g}(\bm{x}_{M+1})]
	\in \mathbb{R}^{N\times M},&
	\\
	\bm{U}:=&[ \bm{u}_1\cdots \bm{u}_M ]
	\in \mathbb{R}^{p\times M}.&
\end{align}

By these notations the observable functions are represented as follows:
\begin{align}
	[\bm{g}(\bm{x}_1,\bm{u}_1)
	\cdots \bm{g}(\bm{x}_M,\bm{u}_M)]
	=&
	\left[
		\begin{array}{c}
			\bm{G}
		\\
			\bm{U}
		\end{array}
	\right].&
\end{align}

Then the relationship corresponding to \eqref{eq. residual 1} is expressed as:
\begin{align}
	\bm{\hat{G}}=&
	\bm{\hat{\mathcal{K}}}
	\left[
	\begin{array}{c}
	\bm{G}
	\\
	\bm{U}
	\end{array}
	\right]
	+
	[\bm{r}(\bm{x}_1,\bm{u}_1)\cdots \bm{r}(\bm{x}_M,\bm{u}_M)],
	&
	\label{eq. residual 2}
\end{align}
where $\bm{\hat{\mathcal{K}}}\in \mathbb{R}^{N\times (N+p)}$.
Note that the last $p$ rows of $\bm{\mathcal{K}}$
in \eqref{eq. residual 1} are disregarded and instead $\bm{\hat{\mathcal{K}}}$ is introduced
since we are only interested in the prediction of the observable
functions $g_i$.
Minimizing the residual we seek a solution of a least-squares problem:
\begin{align}
	\underset{\bm{\hat{\mathcal{K}}}}{\text{min}}
	\left\| 
	[\bm{r}(\bm{x}_1,\bm{u}_1)\cdots \bm{r}(\bm{x}_M,\bm{u}_M)]
	\right\|_F^2.
\end{align}

The analytical expression of the solution to this problem can be described as follows:
\begin{align}
	[\bm{A}\ \bm{B}]:=
	\bm{\hat{\mathcal{K}}}=
	\bm{\hat{G}}
	\left[
		\begin{array}{c}
			\bm{G}
		\\
			\bm{U}
		\end{array}
	\right]^\dagger,
	\label{eq. linear predictor pseudo}
\end{align}
where $\dagger$ denotes pseudoinverse, and the matrices $\bm{A}$ and $\bm{B}$ correspond to those in \eqref{eq. linear predictor}.

\section{CONTOLLER SYNTHESIS USING THE $\mathcal{H}_2$ NORM CONDITION}
\label{sec3}
In the proposed method, we design a static feedback controller gain $\bm{S}$ s.t.
\begin{align}
	\label{eq. controller concept}
	&\bm{u}_{k} \hspace{-0.5mm}=\hspace{-0.5mm}
	 \bm{T}(\bm{x}_k) \hspace{-0.5mm}=\hspace{-0.5mm}
	 \bm{S}\bm{g}(\bm{x}_k),
	 \ 
	\left(
		\bm{T}\hspace{-0.5mm}:\hspace{-0.5mm}
		\mathcal{X}\rightarrow \mathcal{U},\ 
		\bm{S}\hspace{-0.5mm} \in \hspace{-0.5mm} 
		\mathbb{R}^{p\times N}
	\right).
	&
\end{align}

Note that the controller may have a nonlinear structure in terms of the original state $\bm{x}_k$ while it is a linear feedback controller in terms of the observable functions $\bm{g}(\bm{x}_k)$. We design controllers in the observable functions domain, which allows us to make use of the linear control theory
in nonlinear control problems.

\subsection{Model Uncertainty of the Data-Driven Control Models}
The predictive accuracy of the linear control models obtained
in Section \ref{sec2}
depends on the norm of the residual $\bm{r}$ in \eqref{eq. residual 2},
and the models capture the behavior of the underlying
systems with no errors if $\bm{r}\equiv \bm{0}$.
Nevertheless this does not hold in general due to the nature of the data-driven setting, i.e., it is difficult in real applications to satisfy the condition $\bm{r}\equiv \bm{0}$
by controlling several factors related to the data-driven procedure:
the choice of observable functions, the amount of collected 
data, and so on.
%i.e. several factors related to the data-driven procedure, the choice of observable functions, the amount of data collected, and so on, directly depend on the predictive accuracy of the control models. 
As a result the prediction error of the linear control model
\eqref{eq. linear predictor} may lower the performance of data-driven controllers.
%the residual leads to the prediction error of the linear model \eqref{eq. linear predictor}, which may lower the performance of controllers. 
In this paper we explicitly deal with this model uncertainty
by incorporating the parameter $\bm{\rho}$ into the system matrices:
%\vspace{-2mm}
\begin{align}
	\bm{g}(\bm{x}_{k+1}) = 
	\bm{A}(\bm{\rho})\bm{g}(\bm{x}_{k}) + \bm{B}(\bm{\rho})\bm{u}_k,
	\label{eq. nominal model without the disturbance}
\end{align}
where $\bm{A}(\bm{\rho})$ and $\bm{B}(\bm{\rho})$
depend on $\bm{\rho}$, which represents the variation of the property of the model due to
the variation of the factors related to the data-driven procedure in Section \ref{sec2}.
We treat the parameter $\bm{\rho}$
not explicitly but implicitly through data sets we collect.
As an example, suppose that we deal with a dynamical system and have a data set $D_1$. 
Using $D_1$ we can construct a data-driven model
$(\bm{A}_1,\bm{B}_1)$, to which
some value of $\bm{\rho}=\bm{\rho}_1$ corresponds, or
$\bm{A}_1=\bm{A}(\bm{\rho}_1),\ \bm{B}_1=\bm{B}(\bm{\rho}_1)$.
Then if we collect another data set $D_2$ another data-driven
model $(\bm{A}_2,\bm{B}_2)$, to which
some value of $\bm{\rho}=\bm{\rho}_2$ corresponds,
can be constructed, where
$\bm{A}_2=\bm{A}(\bm{\rho}_2),\ \bm{B}_2=\bm{B}(\bm{\rho}_2)$.
Note that in many cases $\bm{\rho_1}\neq \bm{\rho_2}$
due to the model uncertainty caused by the factors related to the data-driven setting.
In this way we implicitly deal with the parameter $\bm{\rho}$
by constructing multiple control models \eqref{eq. linear predictor} from multiple data sets. 
It is emphasized that we have no access
to the parameter $\bm{\rho}$ online, while the model \eqref{eq. nominal model without the disturbance} has an structure which is equivalent to linear parameter-varying (LPV) systems\cite{LPV_reference}, where the parameter $\bm{\rho}$
can usually be measured online. 
Also note that there is a study which aims to identify the Koopman operator
of \textit{unforced} systems in the presence of actuation
based on the concept of LPV-modeling\cite{Koopman_to_actuated_systems},
while the purpose of our proposed method is to design data-driven controllers.

In addition to the model uncertainty derived from the parameter $\bm{\rho}$,
we account for the effect of the disturbance $\bm{w}_k$ to the system
by introducing an additional term:
%\vspace{-2mm}
\begin{align}
\bm{g}(\bm{x}_{k+1}) = 
\bm{A}(\bm{\rho})\bm{g}(\bm{x}_{k}) + \bm{B}(\bm{\rho})\bm{u}_k
+
\bm{B}_w(\bm{\rho}) \bm{w}_k.
\label{eq. linear predictor 2}
\end{align}

If we have access to the trajectory of $\bm{w}_k$,
$\bm{B}_w(\bm{\rho})$ can also be estimated by the slight modification to the procedure in Section \ref{sec2}
\cite{KORDA_Koopman_MPC}. When we have no access to the
measurement of $\bm{w}_k$, other estimation
techniques\cite{Surana_2016_CDC,Surana_2016_IFAC} may be used to estimate $\bm{B}_w(\bm{\rho})$ from data.
For simplicity we set 
$\bm{B}_w(\bm{\rho})\equiv [1\cdots 1]^\tr$
in all numerical examples in Section \ref{sec4}.
\vspace{-2mm}
\subsection{Construction of Polytope Sets}
Using the representation \eqref{eq. linear predictor 2}
we model our system as a polytope set which consists of $N_p$ vertices:
\begin{align}
&\left[
\bm{A}(\bm{\rho})\ \bm{B}(\bm{\rho})\
\bm{B}_w(\bm{\rho})
\right]
=
\sum_{i=1}^{N_p} \alpha_i 
\left[
\bm{\tilde{A}}_i\ \bm{\tilde{B}}_i\
\bm{\tilde{B}}_{wi}
\right],
&
\label{eq. polytope set}
\end{align}
where $\alpha_i\geq 0$, $\sum_{i=1}^{N_p} \alpha_i = 1$. 
The main motivation for interpolating the system by a polytope set is that robust
controllers for nonlinear systems can be designed
by convex optimization on the basis of LMI conditions,
which is described in Section \ref{section controller synthesis}.
Specifically this polytope set is formed by multiple data-driven linear models \eqref{eq. linear predictor 2} as follows.
First we collect $N_D$ data sets $\{ D_l \}_{l=1}^{N_D}$, which are then used to
construct the same number of data-driven models
$\{(\bm{A}_l,\bm{B}_l,\bm{B}_{wl})\}_{l=1}^{N_D}$. 
The polytope set \eqref{eq. polytope set} is formed by searching for the maximum and minimum values of each entry of 
$\bm{A}_l$, $\bm{B}_l$, and $\bm{B}_{wl}$.
As a simple example suppose that we have 
$\bm{A}_l=(a_{ij}^l)\in \mathbb{R}^{2\times 2}$, where
$a_{11}^l\hspace{-1mm}\in\hspace{-1mm} [a_{\text{min}}, a_{\text{max}}]$,
$a_{\text{min}}\hspace{-1mm}:=\hspace{-2.5mm}\underset{1\leq l\leq N_D}{\text{min}}a_{11}^l$,
$a_{\text{max}}\hspace{-1mm}:=\hspace{-2.5mm}\underset{1\leq l\leq N_D}{\text{max}}a_{11}^l$,
$a_{12}^l\hspace{-1mm}\equiv\hspace{-1mm} b$,
$a_{21}^l\hspace{-1mm}\in\hspace{-1mm} [c_{\text{min}}, c_{\text{max}}]$, 
$c_{\text{min}}\hspace{-1mm}:=\hspace{-2.5mm}\underset{1\leq l\leq N_D}{\text{min}}a_{21}^l$,
$c_{\text{max}}\hspace{-1mm}:=\hspace{-2.5mm}\underset{1\leq l\leq N_D}{\text{max}}a_{21}^l$, and
$a_{22}^l\hspace{-1mm}\equiv\hspace{-1mm} d$.
Noticing that
\begin{align}
	&
	a_{11}^l\hspace{-1mm}\in\hspace{-1mm} [a_{\text{min}}, a_{\text{max}}],
	\ \
	a_{21}^l\hspace{-1mm}\in\hspace{-1mm} [c_{\text{min}}, c_{\text{max}}]
	&\nonumber
\\ 
	\Leftrightarrow\
	&
	\left\{
		\begin{array}{l}
			a_{11}^l=\beta_1a_{\text{min}}+\beta_2a_{\text{max}},
			\ \beta_i\geq 0,\ \beta_1+\beta_2=1,
		\\
			a_{21}^l=\gamma_1c_{\text{min}}+\gamma_2c_{\text{max}},
			\ \gamma_i\geq 0,\ \gamma_1+\gamma_2=1,
		\end{array}
	\right.
	&\nonumber
\end{align}
it is shown that
\begin{align}
	a_{11}^l&=
	\left(
		\beta_1a_{\text{min}}+\beta_2a_{\text{max}}
	\right)
	\left(
		 \gamma_1+\gamma_2
	\right)
	&\nonumber
\\
	&
	=\hspace{-0.8mm}
	\underset{=:\alpha_1}{\underline{\beta_1\gamma_1}}a_{\text{min}}
	\hspace{-0.5mm}+\hspace{-0.5mm}
	\underset{=:\alpha_2}{\underline{\beta_1\gamma_2}}a_{\text{min}}
	\hspace{-0.5mm}+\hspace{-0.5mm}
	\underset{=:\alpha_3}{\underline{\beta_2\gamma_1}}a_{\text{max}}
	\hspace{-0.5mm}+\hspace{-0.5mm}
	\underset{=:\alpha_4}{\underline{\beta_2\gamma_2}}a_{\text{max}}
	&\nonumber
\\ 
	&
	=
	\alpha_1a_{\text{min}}+
	\alpha_2a_{\text{min}}+
	\alpha_3a_{\text{max}}+
	\alpha_4a_{\text{max}},
	&
\\
	a_{21}^l&=
	\alpha_1c_{\text{min}}+
	\alpha_2c_{\text{min}}+
	\alpha_3c_{\text{max}}+
	\alpha_4c_{\text{max}},
	&
\\
	\alpha_1&+\alpha_2+\alpha_3+\alpha_4
	=
	\left(
	\beta_1+\beta_2
	\right)
	\left(
	\gamma_1+\gamma_2
	\right)
	=
	1.&
\end{align}
$\therefore$ A polytope set is derived as:
\begin{align}
&
	\bm{A}(\bm{\rho})
	=
	\left[
		\begin{array}{lr}
		\rho_1&b
		\\
		\rho_2&d
		\end{array}
	\right]
	=\sum_{i=1}^{4}\alpha_i \bm{\tilde{A}}_i,
	\ \ 
	\bm{\tilde{A}}_1 \hspace{-1mm} := \hspace{-1mm}
	\left[
	\begin{array}{lr}
	\hspace{-1.5mm}a_{\text{max}}\hspace{-1.5mm} & \hspace{-1.5mm}b\hspace{-1.5mm}
	\\
	\hspace{-1.5mm}c_{\text{max}}\hspace{-1.5mm} & \hspace{-1.5mm}d\hspace{-1.5mm}
	\end{array}
	\right],
	&\nonumber
\\
	&
	\bm{\tilde{A}}_2 \hspace{-1mm} := \hspace{-1mm}
	\left[
	\begin{array}{lr}
	\hspace{-1.5mm}a_{\text{max}}\hspace{-1.5mm} & b\hspace{-1.5mm}
	\\
	\hspace{-1.5mm}c_{\text{min}}\hspace{-1.5mm} & d\hspace{-1.5mm}
	\end{array}
	\right]\hspace{-1mm},
	\bm{\tilde{A}}_3 \hspace{-1mm} := \hspace{-1mm}
	\left[
	\begin{array}{lr}
	\hspace{-1.5mm}a_{\text{min}}\hspace{-1.5mm} & b\hspace{-1.5mm}
	\\
	\hspace{-1.5mm}c_{\text{min}}\hspace{-1.5mm} & d\hspace{-1.5mm}
	\end{array}
	\right]\hspace{-1mm},
	\bm{\tilde{A}}_4:=
	\left[
	\begin{array}{lr}
	\hspace{-1.5mm}a_{\text{min}}\hspace{-1.5mm} & b\hspace{-1.5mm}
	\\
	\hspace{-1.5mm}c_{\text{max}}\hspace{-1.5mm} & d\hspace{-1.5mm}
	\end{array}
	\right]\hspace{-1mm},
	&
\end{align}
which means that the model
represents all the plants in a rectangle in the parameter
space as shown in Fig.~\ref{fig. rectangle}.
\begin{figure}[t]
	\centering
	\includegraphics[width=6cm]{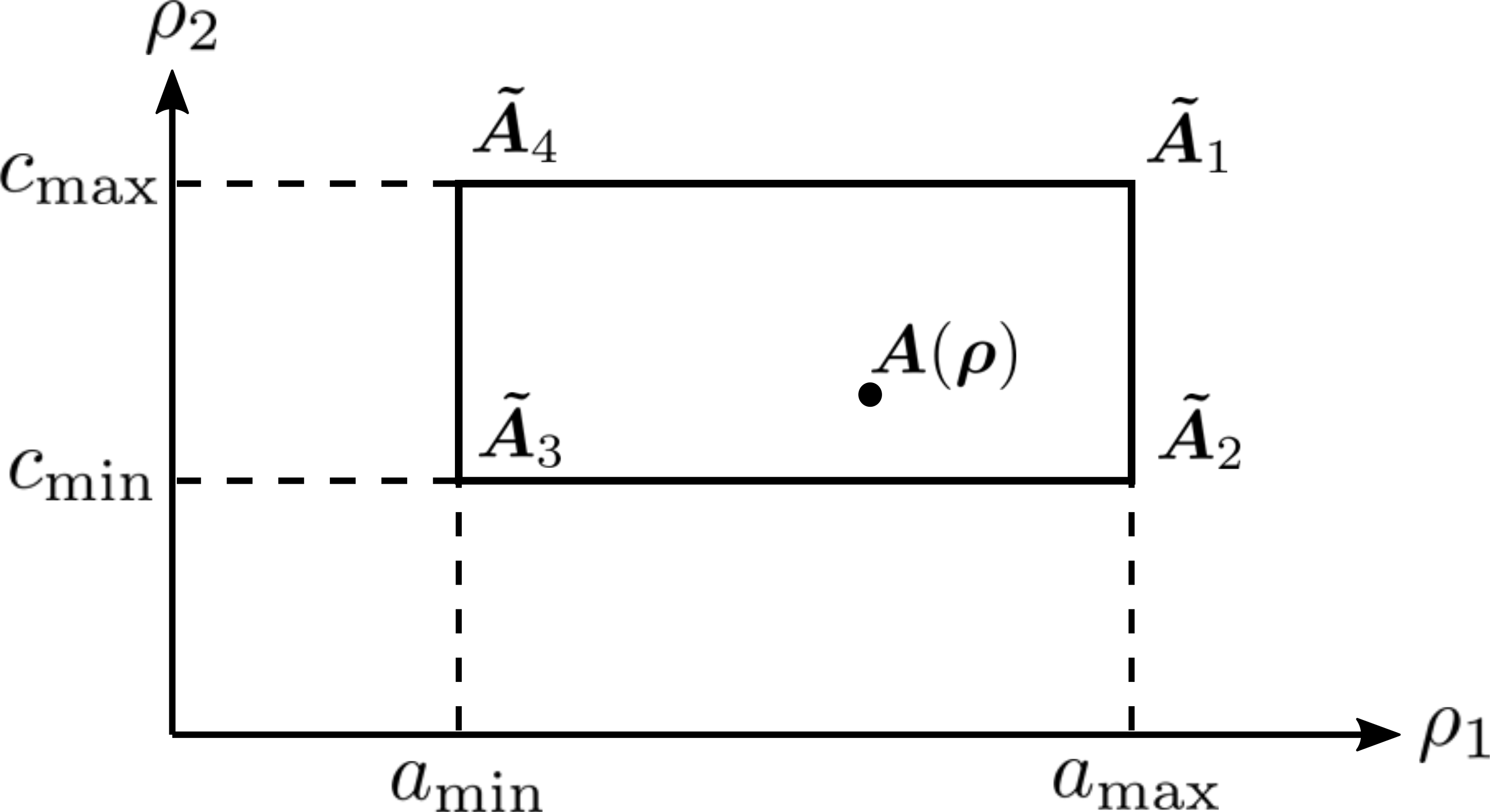}
	\caption{Polytope set represented in the parameter space.}
	\label{fig. rectangle}
\end{figure}
Other matrices $\bm{\tilde{B}}_l$ and $\bm{\tilde{B}}_{wl}$
are formed in the same way.
Note that it is important to use a sufficient number of observable functions
with appropriate design to avoid ill-conditioning which leads to an extremely
large size of the polytope.
\vspace{-2mm}

\subsection{Construction with a Threshold Algorithm}
%\vspace{-2mm}
In a practical point of view the construction of a polytope set
\eqref{eq. polytope set} may require prohibitively high
computational storage since it is necessary to evaluate
$2^{N^2}$ vertices.
%For instance, if we design ten observable functions ($N=10$)
%we need to store $1.2677\times 10^{30}$ matrices for the formulation of the polytope set.
Thus in order to construct a polytope set
with a moderate number of vertices we apply a threshold algorithm described in Algorithm \ref{alg. threshold}. We generate $2^h$ vertices considering not all entries but only $h$ entries whose variations within the data sets are larger than others, which
are fixed to the mean values.
%We can set the threshold number $h$ by investigating the variations
%of the system matrices $\{(\bm{A}_l,\bm{B}_l,\bm{B}_{wl})\}_{l=1}^{N_D}$ and ignoring entries
%whose variation is negligibly small.

\subsection{Extended $\mathcal{H}_2$ Norm Characterization}
	With respect to the data-driven model \eqref{eq. linear predictor 2}, we define the controlled output
	$\bm{z}_k\in \mathbb{R}^d$
	as follows:
	\begin{align}
	\bm{z}_k
	=
	\bm{C}_z
	\bm{g}(\bm{x}_k)
	+
	\bm{D}_{zu} \bm{u}_k.
	\label{eq. controlled output}
	\end{align}
    In the proposed controller synthesis, the cost function $J_{\text{obj}}$ to be minimized
    is defined as:
	\begin{align}
		J_{\text{obj}}:=\sum_{k=1}^{\infty} \bm{z}_k^\tr \bm{z}_k,
	\end{align}
	whose expected value is equivalent to the $\mathcal{H}_2$ norm of the model provided $\bm{w}_k$ is the white noise.
Then the extended $\mathcal{H}_2$ norm characterization of
the generalized plant 
(\eqref{eq. linear predictor 2}, \eqref{eq. controlled output}) is described as follows\cite{Extended_H2_H_inf}.
\begin{thm}
	\label{thm. 1}
	\begin{align}
		&
		\hspace{7mm}
		\left\|  \mathcal{H}_{wz} \right\| _2^2 < \mu 
		\ \ \Leftrightarrow\ \ 
		\exists 
		\bm{P},\,\bm{W},\,
		\bm{X},\,
		\bm{L}\ 
		\text{s.t.}
		&\nonumber
	\\
		&
		\text{tr}(\bm{W})<\mu,
		&
	\\
		&
		\bm{M}_1 \hspace{-0.5mm}:=\hspace{-0.5mm}
		\left[
			\begin{array}{cc}
				\bm{W}&\bm{C}_z\bm{X}+\bm{D}_{zu}\bm{L}
			\\
				*&\bm{X}+\bm{X}^\tr -\bm{P}
			\end{array}
		\right]>0,
		&
	\\
		&
		\bm{M}_2 \hspace{-0.5mm}:=\hspace{-0.5mm}
		\left[
		\begin{array}{ccc}
			\bm{P}&\bm{A}(\bm{\rho})\bm{X}+\bm{B}(\bm{\rho})\bm{L} & \bm{B}_w(\bm{\rho})
		\\
			*&\bm{X}+\bm{X}^\tr -\bm{P} & \bm{0}
		\\
			*&*&\bm{I}
		\end{array}
		\right]>0,
		&
	\end{align}
	where $\bm{P}$ and $\bm{W}$ are symmetric, and $\left\|  \mathcal{H}_{wz} \right\| _2$ denotes
	the $\mathcal{H}_2$ norm from $\bm{w}_k$ to $\bm{z}_k$.
\end{thm}

With variables satisfying this theorem, a static feedback controller of the form \eqref{eq. controller concept}
that guarantees the stability of the closed system and satisfies $\left\|  \mathcal{H}_{wz} \right\| _2^2 < \mu $ can be obtained as 
$\bm{S}=\bm{L}\bm{X}^{-1}$
\cite{Extended_H2_H_inf}.
%\begin{align}
%	\bm{S}=\bm{L}\bm{X}^{-1}.
%\end{align}
\begin{algorithm}[t]
	\caption{Construction of a polytope set}
	\label{alg. threshold}
	\begin{algorithmic}
		\Require  
		Matrices $\{\bm{A}_l\}_{l=1}^{N_D}$ 
		and the threshold number $h\in \mathbb{N}$
		\Ensure A set $\{ \bm{\tilde{A}}_i \}_{i=1}^{2^h}$ of a polytope set
		\State $\bm{A}_{\text{max}}(i,j)\leftarrow 
		\underset{1\leq l\leq N_D}{\text{max}}[\bm{A}_l(i,j)]$ for $\forall$ $i,j$
		\State $\bm{A}_{\text{min}}(i,j)\leftarrow 
		\underset{1\leq l\leq N_D}{\text{min}}[\bm{A}_l(i,j)]$ for $\forall$ $i,j$
		\State $\bm{A}_{\text{mean}}(i,j)\leftarrow
		\frac{1}{N_D}\sum_{l=1}^{N_D}
		\bm{A}_l(i,j)$ for $\forall$ $i,j$
		\State $\bm{\tilde{A}}_1\leftarrow
		\bm{A}_{\text{mean}}$
		\State $\bm{\tilde{A}}_2\leftarrow
		\bm{A}_{\text{mean}}$
		\For{$p=1:h$} \vspace{2mm}
		\State \vspace{2mm}
		$(s,t)\leftarrow$ index of the
		$p$-th largest entry of $\bm{A}_{\text{max}}-\bm{A}_{\text{min}}$
		\For {$k=1:2^{p-1}$}
		\State $\bm{\tilde{A}}_{2^{p-1}+k}\leftarrow
		\bm{\tilde{A}}_k$
		\State $\bm{\tilde{A}}_k(s,t)\leftarrow
		\bm{A}_{\text{max}}(s,t)$
		\State $\bm{\tilde{A}}_{2^{p-1}+k}(s,t)\leftarrow
		\bm{A}_{\text{min}}(s,t)$
		\EndFor
		\EndFor
		%\State (Note that other matrices $\bm{\tilde{B}}_l$
		%and $\bm{\tilde{B}}_{wl}$ are constructed in the same way.)
	\end{algorithmic}
\end{algorithm}
%\vspace{-7mm}

\subsection{Robust Controller Synthesis} 
\label{section controller synthesis}
%\vspace{-7mm}
Using the vertex matrices 
$\{ \bm{\tilde{A}}_i,\bm{\tilde{B}}_i,\bm{\tilde{B}}_{wi} \}$
of the polytope set \eqref{eq. polytope set}
and the controlled output \eqref{eq. controlled output}, 
we solve the following problem to design a robust controller
that accounts for the model uncertainty due to the
the data-driven procedure.
\begin{pbm}
	\rm{(Proposed Controller Synthesis)}
	\ \\
	Given vertex matrices $\{ \bm{\tilde{A}}_i,\bm{\tilde{B}}_i,\bm{\tilde{B}}_{wi} \}_{i=1}^{2^h}$, $\bm{C}_z$, and $\bm{D}_{zu}$,
	solve the following problem:
	\begin{align}
	&
	\text{inf tr}(\bm{W}_i)
	\ \text{  subject to }
	&\nonumber
	\\
	&
	\bm{M}_1 
	(\bm{P},\bm{W},\bm{X},\bm{L};\
	\bm{C}_z, \bm{D}_{zu}
	)
	>0,
	&\nonumber
	\\
	&
	\bm{M}_2(\bm{P},\bm{X},\bm{L};\ 
	\bm{\tilde{A}}_i,\bm{\tilde{B}}_i,\bm{\tilde{B}}_{wi}) 
	>0,\ \ \ 
	\text{for }i=1,\cdots,2^h,
	&\nonumber
	\end{align}
	and define 
	$J_{\text{syn}}:=\underset{1\leq i\leq 2^h}{\text{max}}\{ \text{tr}(\bm{W}_i) \}$
	and
	$\bm{S}:=\bm{L}\bm{X}^{-1}$.
	\label{pbm. 1}
\end{pbm}

We also define the following problem for evaluating the bound $J(\bm{\rho})$
for the $\mathcal{H}_2$ norm.
\begin{pbm}
	\label{pbm. analysis}
	\rm{(Evaluation of the $\mathcal{H}_2$ norm)}
	\ \\
	Given system matrices $\bm{A}(\bm{\rho})$, $\bm{B}(\bm{\rho})$,
	$\bm{B}_w(\bm{\rho})$, $\bm{C}_z$, $\bm{D}_{zu}$, and a feedback gain $\bm{S}$, solve the following problem:
	\begin{align}
	&
	\text{inf tr}(\bm{W}) 
	\ \text{  subject to }
	&
	\\
	& 
	\bm{M}_3
	\hspace{-0.5mm}:=\hspace{-0.5mm}
	\left[
	\begin{array}{cc}
	\bm{W}&
	\left(
	\bm{C}_z+\bm{D}_{zu}\bm{S}
	\right)\bm{X}
	\\
	*&\bm{X}+\bm{X}^\tr -\bm{P}
	\end{array}
	\right]>0,
	&
	\\
	&
	\bm{M}_4
	\hspace{-0.5mm}:=\hspace{-0.5mm}
	\left[
	\begin{array}{ccc}
	\bm{P}&
	\left(
	\bm{A}(\bm{\rho})+\bm{B}(\bm{\rho})\bm{S}
	\right)\bm{X} 
	& \bm{B}_w(\bm{\rho})
	\\
	*&\bm{X}+\bm{X}^\tr -\bm{P} & \bm{0}
	\\
	*&*&\bm{I}
	\end{array}
	\right]>0,
	&
	\end{align}
	and define $J(\bm{\rho}):=\text{tr}(\bm{W})$.
\end{pbm}
\begin{col}
	\label{col. 2}
	\ \\
	\rm{}
	If we solve Problem \ref{pbm. analysis} using a feedback gain $\bm{S}$ determined by Problem \ref{pbm. 1}, then
	$J(\bm{\rho})\leq J_{\text{syn}}$ for $\forall \bm{\rho}$.
\end{col}
\begin{proof}
	It is easy to prove the proposition, and the proof is omitted
	due to lack of space.
\end{proof}

By Corollary 1 it is confirmed that the proposed controller guarantees that
for all the plants in the polytope set \eqref{eq. polytope set}
the $\mathcal{H}_2$ norm is bounded by $J_{\text{syn}}$.
\vspace{-2mm}

\subsection{Relation Between the Model Uncertainty and the $\mathcal{H}_2$ Norm}
We can relate the model uncertainty of the model, which corresponds to
the parameter $\bm{\rho}$, and the bound for the $\mathcal{H}_2$ norm as follows.
\begin{thm}
	\ \\
	Let $\bm{A}(\bm{\rho})$, $\bm{B}(\bm{\rho})$, and
	$\bm{B}_w(\bm{\rho})$ be system matrices. 
	Consider two sets $\mathcal{O}:=\{ \bm{\rho} \}$ and 
	$\mathcal{O}':=\{ \bm{\rho}' \}$ of parameters, where
	$J_{\text{syn}}$ and $J_{\text{syn}}'$ denote
	the corresponding bounds for the $\mathcal{H}_2$ norm determined by Problem \ref{pbm. 1}, respectively.
	If $\mathcal{O}\subseteq \mathcal{O}'$, then 
	$J_{\text{syn}}\leq J_{\text{syn}}'$.
\end{thm}
\begin{proof}
	It is easily shown that if $\mathcal{O}\subseteq \mathcal{O}'$,
	\begin{align}
		&
		\label{eq. proof systems subseteq}
		\left\{
		\left[
		\bm{A}(\bm{\rho})\bm{B}(\bm{\rho})\bm{B}_w(\bm{\rho})
		\right]
		\right\}
		\subseteq
		\left\{
		\left[
		\bm{A}(\bm{\rho}')\bm{B}(\bm{\rho}')\bm{B}_w(\bm{\rho}')
		\right]
		\right\},
		&
	\end{align}
	for all $\bm{\rho}$ and $\bm{\rho}'$. 
	Let $\Sigma:=\{ [\bm{\tilde{A}}_i\,\bm{\tilde{B}}_i\,\bm{\tilde{B}}_{wi}] \}$ and
	$\Sigma':=\{ [\bm{\tilde{A}}_i'\,\bm{\tilde{B}}_i'\,\bm{\tilde{B}}_{wi}'] \}$ be
	vertex matrices that correspond to $\mathcal{O}$
	and $\mathcal{O}'$, respectively.
	By \eqref{eq. proof systems subseteq} and Corollary \ref{col. 2} it is shown that
	for any $\bm{\rho}\in \mathcal{O}$,
	\begin{align}
	J(\bm{\rho})\leq J_{\text{syn}}',
	\label{eq. proof rho and J_syn'}
	\end{align}
	where $J_{\text{syn}}'$ denotes the bound determined by Problem \ref{pbm. 1} with the vertices $\Sigma'$.
	On the other hand,
	by Corollary \ref{col. 2} there exists $\bm{\rho}_I\in \mathcal{O}$
	s.t. $\bm{\rho}_I$
	corresponds to one of the vertices $\Sigma$
	and  
	\begin{align}
	J(\bm{\rho}_I)=J_{\text{syn}},
	\label{eq. proof J_rho and J_syn}
	\end{align}
	where $J_{\text{syn}}$ denotes the bound determined by Problem \ref{pbm. 1} with the vertices $\Sigma$.
	Substituting $\bm{\rho}=\bm{\rho}_I$ into
	\eqref{eq. proof rho and J_syn'}, 
	\begin{align}
	J_{\text{syn}}=J(\bm{\rho}_I)\leq J_{\text{syn}}'
	\ \Leftrightarrow\
	J_{\text{syn}}\leq J_{\text{syn}}'.
	\end{align}
\end{proof}

Theorem 2 states that the \textit{worst case} $\mathcal{H}_2$ norm,
which corresponds to $J_{\text{syn}}$,
becomes greater if the model uncertainty becomes greater. 
On the other hand,
it is emphasized
that there may be a gap between $J_{\text{syn}}$ and $J(\bm{\rho})$, i.e., the worst case $\mathcal{H}_2$ norm
and the \textit{actual} $\mathcal{H}_2$ norm.

\section{NUMERICAL EXAMPLES}
\label{sec4}
In this section we provide simulation results of the proposed controller synthesis applied to nonlinear systems.
In addition to the proposed controller we design two other controllers for comparison. One is the LQR regulator designed
for a single data-driven model \eqref{eq. linear predictor 2},
and the other is the nominal $\mathcal{H}_2$ feedback controller,
which we design for a single data-driven model \eqref{eq. linear predictor 2} by solving the following problem:
\begin{pbm}
	(Nominal Controller Synthesis)
	\begin{align}
	&
	\text{inf.    }\ \ \text{tr}(\bm{W})
	\ \text{  subject to }
	&\nonumber
	\\
	&
	\bm{M}_1(\bm{P},\bm{W},\bm{X},\bm{L};\
	\bm{C}_z, \bm{D}_{zu}
	)>0,
	&\nonumber
	\\
	&
	%\hspace{10mm}
	\bm{M}_2(\bm{P},\bm{X},\bm{L};\ 
	\bm{A}(\bm{\rho}),\bm{B}(\bm{\rho}),\bm{B}_{w}(\bm{\rho}))>0
	.
	&
	\label{eq. LMIs 2}
	\end{align}
	\label{pbm. 2}
\end{pbm}
\vspace{-10mm}

\subsection{Duffing Oscillator}
\label{sec. duffing}
\begin{figure}[b]
	\centering
	\includegraphics[width=\columnwidth]{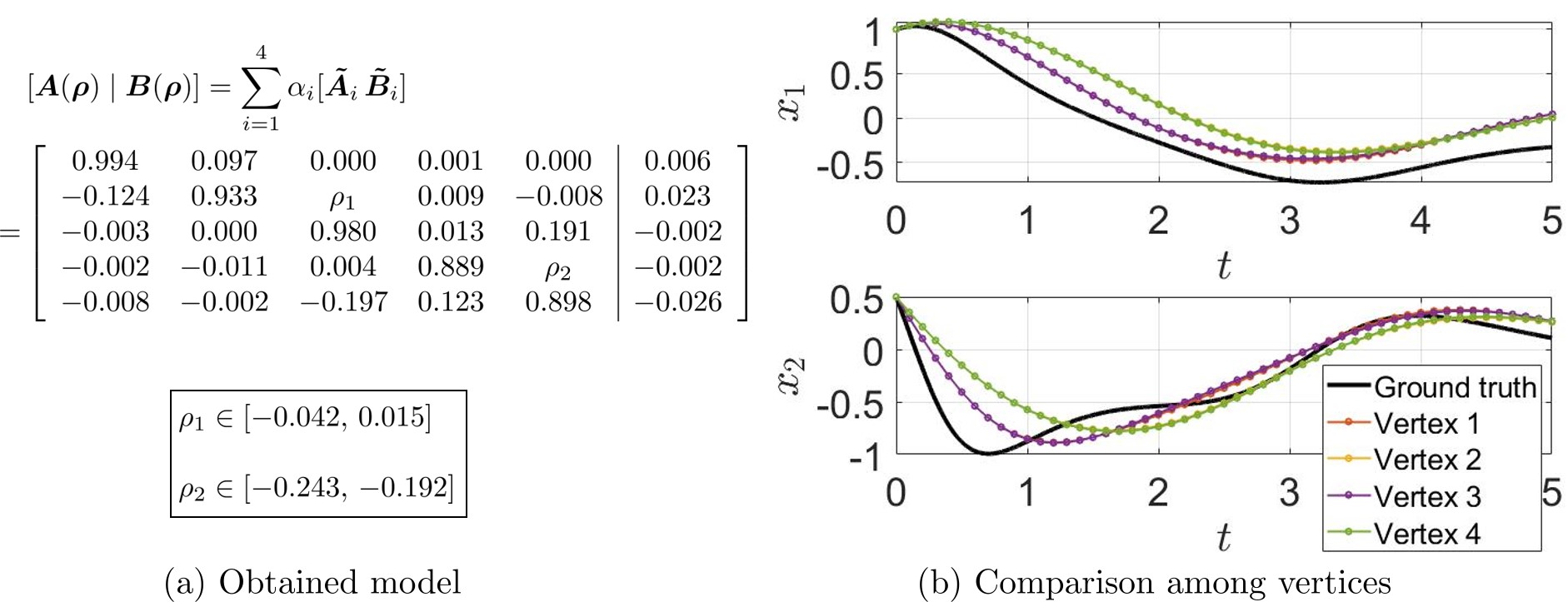}
	\caption{Data-driven model of the duffing oscillator.}
	\label{fig. duffing 1}
\end{figure}
As the first example we consider the forced duffing oscillator:
\begin{align}
	\ddot{x}+0.5\dot{x}-x+4x^3=u(t),
\end{align}
where $u(t)$ represents the input to the system.
We assume that the trajectory of the states $\bm{x}=[x_1\ x_2]^\tr$ ($x_1:=x,\ x_2:=\dot{x}$) is available as data while we have no knowledge about the governing equation. 
In the simulation, measurement noise 
$w \sim \mathcal{N}(w \mid 0,0.01)$ is added to the states.
We specify the observable functions as monomials up to the second order: 
\begin{align}
	\bm{g}(\bm{x}) = 
	[
		x_1\ \ x_2\ \ x_1^2\ \ x_2^2\ \ x_1x_2
	]^\tr.
\end{align}

For the proposed controller synthesis we collect four data sets $\{ D_l \}_{l=1}^4$
with the sampling interval $\Delta t=0.1$.
Each data set consists of 150 pairs $(\bm{F}(\bm{x},u),\bm{x},u)$
of one-step trajectories, where $\bm{x}$ and $u$ are sampled from
the uniform distribution over $[-1,1]$.
In order to stabilize the states while saving the input energy we define the controlled output as
$\bm{z}=
\bm{C}_z
\bm{g}(\bm{x})
+
\bm{D}_{zu} u
=
[10x_1\ x_2\ u]^\tr 
$, where
%\begin{align}
%	\bm{z}=
%	 [10x_1\ x_2\ u]^\tr 
%	=
%	\bm{C}_z
%	\bm{g}(\bm{x})
%	+
%	\bm{D}_{zu} u,
%\end{align}
\begin{align}
\bm{C}_z=
\left[
\begin{array}{c}
\hspace{-4mm}
\begin{array}{cc}
\begin{array}{cc}
10&0
\\
0&1
\end{array}
&
\bm{0}_{2,3}
\hspace{-3mm}
\end{array}
\\
\bm{0}_{1,5}
\hspace{-6mm}
\end{array}
\right],
\bm{D}_{zu}=
\left[
\begin{array}{c}
0
\\
0
\\
1
\end{array}
\right].
\end{align}
Figure \ref{fig. duffing 1} shows the data-driven model obtained
by the proposed method and the comparison of the prediction errors among
the vertices of the polytope. Note that we ignored the dependence of $\bm{\rho}$
on $\bm{B}(\bm{\rho})$ since all the variations of $\bm{B}(\bm{\rho})$
were negligibly small, and we set the threshold number $h=2$ for $\bm{A}(\bm{\rho})$.

The LQR regulator and the nominal $\mathcal{H}_2$ controller
are designed with the single data set $D_1$.
We select the weight matrices for the LQR regulator so that the cost function:
\begin{align}
J_{LQR}=
\sum_{k=1}^\infty 
\left(
\bm{g}(\bm{x}_k)^\tr 
\bm{Q}
\bm{g}(\bm{x}_k)
+
\bm{u}_k^\tr 
\bm{R}
\bm{u}_k
\right),
\end{align}
is equivalent to
the $\mathcal{H}_2$ norm, i.e., we set
\begin{align}
	\bm{Q}=
	\left[
	\begin{array}{c}
	\hspace{-4mm}
	\begin{array}{cc}
	\begin{array}{cc}
	100&0
	\\
	0&1
	\end{array}
	&
	\bm{0}_{2,3}
	\hspace{-3mm}
	\end{array}
	\\
	\bm{0}_{1,5}
	\hspace{-6mm}
	\end{array}
	\right],
	\ \ 
	R=1.
\end{align}

The simulation result is shown in Fig. \ref{fig. duffing 2}.
It is observed that
the proposed robust $\mathcal{H}_2$ controller effectively attenuates the peak of $x_1$ compared to the LQR and the nominal $\mathcal{H}_2$ controllers. As a result
the $l_2$ norm of the proposed controller is smaller than other two controllers (Fig.~\ref{fig. duffing 2} (d)).
This result suggests that the proposed controller successfully dealt with the model uncertainty caused by the data-driven modeling procedure
described in Section \ref{sec2}, while other two controllers
lower the performance in the condition 2 due to this model uncertainty.
\begin{figure}[t]
	\centering
	\includegraphics[width=\columnwidth]{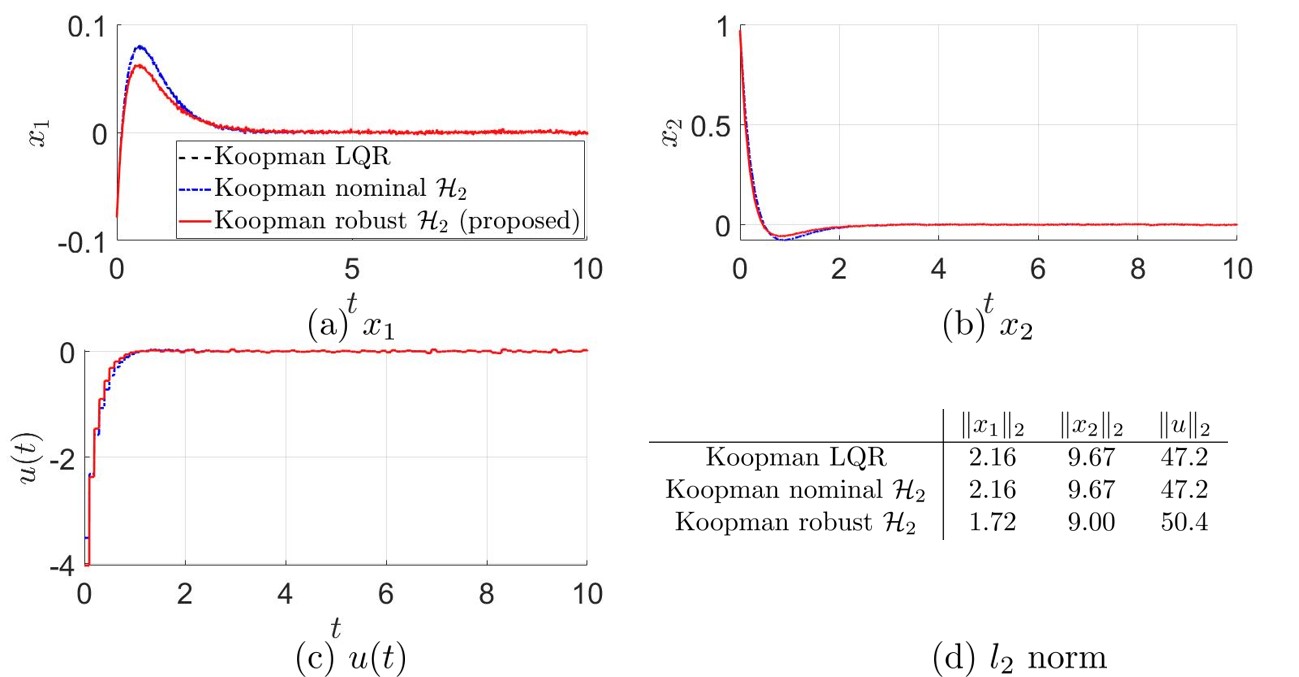}
	\caption{Result of the duffing oscillator with the initial condition: $(x_1(0),x_2(0))=(-0.08,0.97)$.}
	\label{fig. duffing 2}
\end{figure}

\subsection{Shallow-Water Waves Control}
As another example we investigate the controller performance with
the Korteweg-de Vries (KdV) equation, which models the shallow-water waves:
\begin{align}
	\partial_t y(t,x) + y(t,x)\partial_x y(t,x) + \partial_x^3 y(t,x)=u(t,x).
\end{align}
%\begin{comment}

In this example the input 
$\bm{u}_k\hspace{-1mm}=\hspace{-1mm}[u_{1,k} u_{2,k} u_{3,k}]^\tr$
has a structure s.t. 
$u(t,x)=\sum_{i=1}^{3}u_{i,k}
\left\{
U(t,t)-U(t,\Delta t)
\right\}
v_i(x)$,
where
$v_i(x)=e^{ 25\left( x-c_i \right)^2 }$,
$c_1=-\frac{\pi}{2}$, $c_2=0$ $c_3=\frac{\pi}{2}$,
and $U(t,a)$ denotes the step function.
\begin{figure}[b]
	\centering
	\includegraphics[width=\columnwidth]{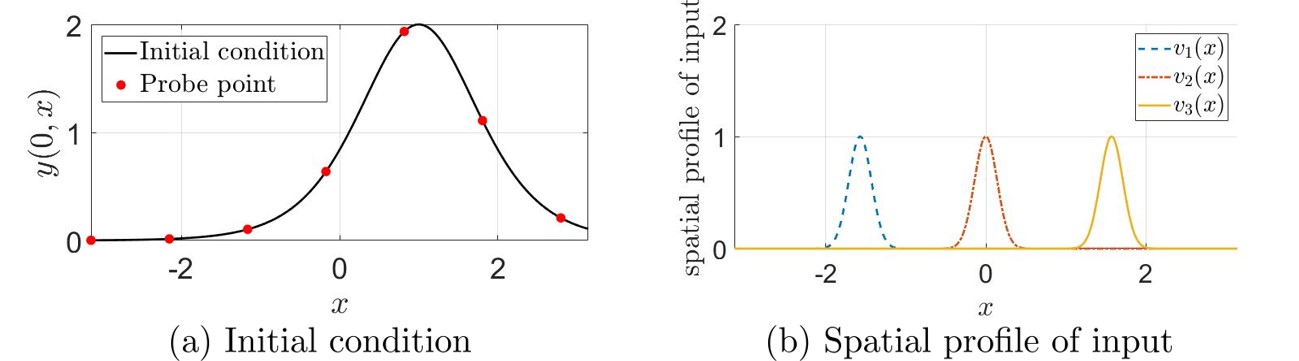}
	\caption{Initial condition and the spatial profile of input.}
	\label{fig. kdv ini}
\end{figure}
\begin{figure}[t]
	\centering
	\includegraphics[width=\columnwidth]{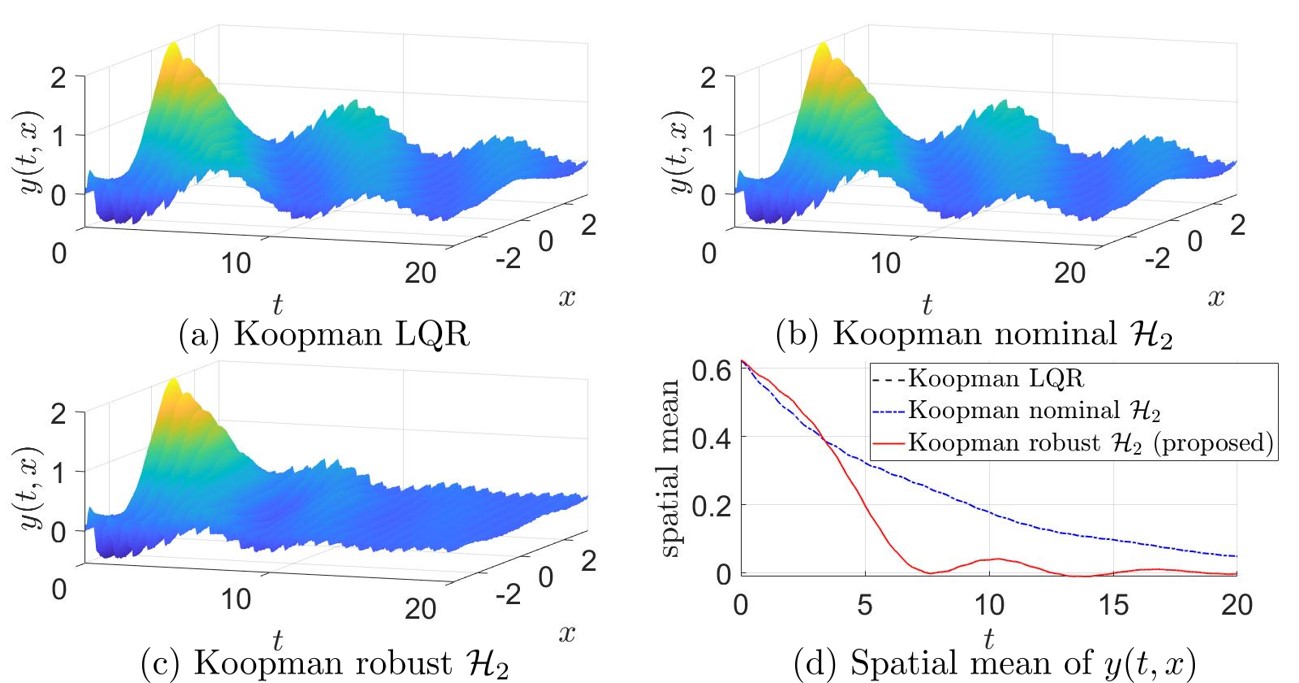}
	\caption{Result of the KdV equation.}
	\label{fig. kdv y}
\end{figure}

In the simulation,
the equation is discretized by the split-stepping method, 
which generates the state space of 128 dimensions with the
time discretization of 0.01 seconds.
With the sampling interval $\Delta t \hspace{-1mm}=\hspace{-1mm}0.01$ we measure the water surface $y(t,x)$ at seven probe points that are shown in Fig.~\ref{fig. kdv ini}~(a), and 
use them as observable functions, which are referred to as $\bm{x}_k$.
We collect four data sets, each of which consists of 100
trajectories with a length of 200 steps,
where
the initial conditions are specified by random convex combinations of three spatial
profiles:
$e^{-(x-\pi/2)^2}$, $-\sin(x/2)^2$, and $e^{-(x+\pi/2)^2}$.
The controlled output is set as
$
\bm{z}_k=
\bm{C}_z \bm{g}(\bm{x}_k)
+
\bm{D}_{zu} \bm{u}_k
=[\bm{x}_k^\tr\ \bm{u}_k^\tr ]^\tr 
$, where
\begin{align}
	\bm{C}_z=[\bm{I}_7\ \bm{0}_{7,3}]^\tr,\ \ 
	\bm{D}_{zu}=[\bm{0}_{3,7}\ \bm{I}_3 ]^\tr.
\end{align}

%$
%\bm{C}_z=[\bm{I}_7\ \bm{0}_{7,3}]^\tr
%$, and
%$
%\bm{D}_{zu}=[\bm{0}_{3,7}\ \bm{I}_3 ]^\tr
%$.
\begin{figure}[t]
	\centering
	\includegraphics[width=\columnwidth]{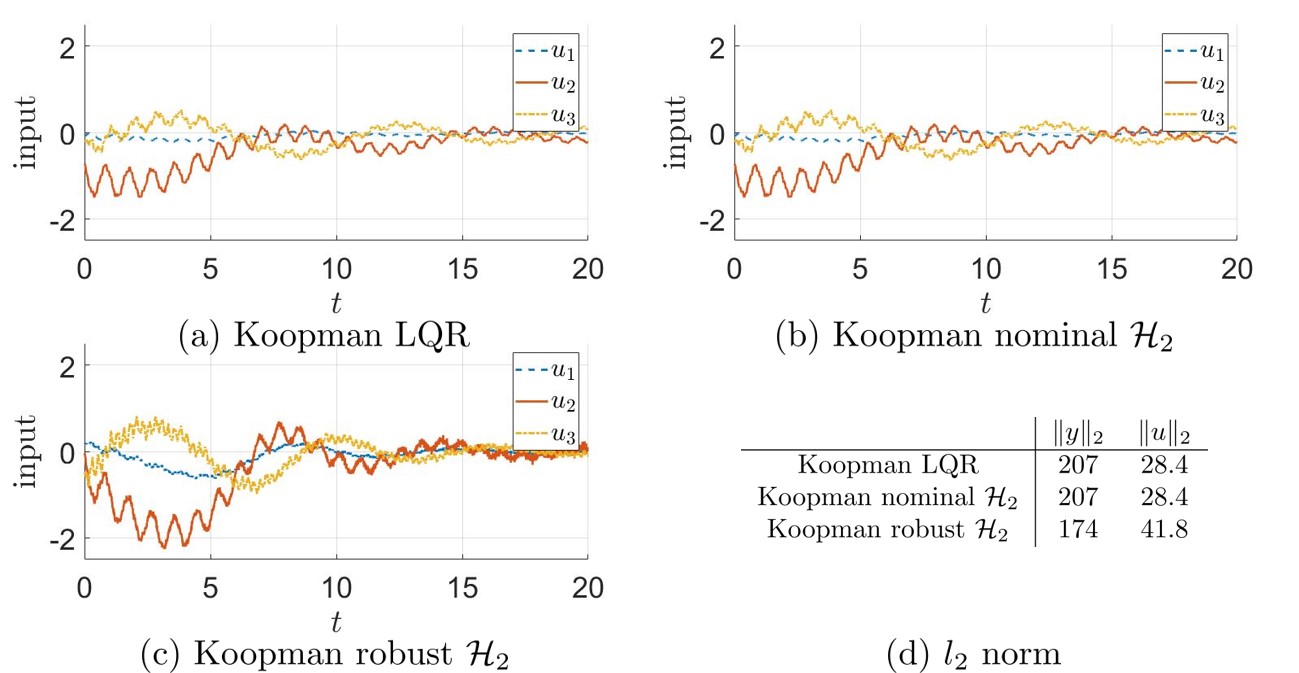}
	\caption{Input obtained by the controllers.}
	\label{fig. kdv u}
\end{figure}

\vspace{-0.5mm}
We set the threshold number $h=2$ for $\bm{A}(\bm{\rho})$,
and ignored the dependence of $\bm{\rho}$ on $\bm{B}(\bm{\rho})$
since the variations of entries were negligibly small.
The weight matrices for the LQR controller are defined as
%\begin{align}
$	\bm{Q}=\bm{I}_7,\   \bm{R}=\bm{I}_3,$
%\end{align}
%$\bm{Q}:=\bm{I}_7$, $\bm{R}:=\bm{I}_3,$ 
so that the cost function
$J_{LQR}$ is equivalent to the $\mathcal{H}_2$ norm.

The result is shown Figs.~\ref{fig. kdv y} and \ref{fig. kdv u}.
It is shown that the proposed controller quickly 
stabilized the states
(Fig.~\ref{fig. kdv y}~(c)) while other two controllers had difficulty
regulating the system (Fig.~\ref{fig. kdv y} (a),(b)).
%\vspace{3mm}

\begin{comment}
\begin{figure}
	\begin{subfigmatrix}{2}
		\setlength{\tabcolsep}{-10mm}
		\subfigure[a]{\includegraphics[width=0.45\columnwidth]{fig/kdv/kdv_spatial_mean.jpg}
			\label{fig:figa11}}
		\subfigure[a]{\includegraphics[width=0.45\columnwidth]{fig/kdv/kdv_spatial_mean.jpg}
			\label{fig:figa12}}
		\subfigure[a]{\includegraphics[width=0.45\columnwidth]{fig/kdv/kdv_spatial_mean.jpg}
			\label{fig:figa21}}
		\subfigure[a]{\includegraphics[width=0.45\columnwidth]{fig/kdv/kdv_spatial_mean.jpg}
			\label{fig:figa22}}
	\end{subfigmatrix}
	\caption{a}\label{fig:figa}
\end{figure}
\end{comment}

\section{CONCLUSIONS}
This paper presented a data-driven control synthesis that combines
the data-driven Koopman operator theory and the $\mathcal{H}_2$ characterization
of linear systems. In order to deal with the model 
uncertainty due to the nature of the data-driven modeling, the proposed method models the system as a polytope set, which is then utilized to design a robust feedback controller
on the basis of the $\mathcal{H}_2$ norm characterization. 
Future directions of research include the introduction
of statistical properties to the proposed data-driven model.
%\vspace{3mm}

\bibliographystyle{IEEEtran}  
\bibliography{IEEEabrv, ACC2021_UCHIDA}

\begin{comment}

\end{comment}

\end{document}